\documentclass[12pt]{amsart}
\usepackage{amscd,color}
\usepackage{bbm}
\usepackage{mathrsfs}
\usepackage{bbm}
\usepackage{amsfonts}
\usepackage{amsmath, amssymb, amsthm}
\usepackage{stmaryrd}
\usepackage{graphics}

\usepackage{latexsym,amsfonts,amssymb}

\newtheorem{thm}{Theorem}[section]

\newtheorem{lemma}[thm]{Lemma}
\newtheorem{remark}[thm]{Remark}
\newtheorem{corollary}[thm]{Corollary}

\begin{document}
\baselineskip=16pt

\title[Unitary and orthogonal equivalence of sets of matrices]
{Unitary and orthogonal equivalence of sets of matrices}
\author{Naihuan Jing}
\address{Department of Mathematics, North Carolina State University,
Raleigh, NC 27695-8205, USA}
\email{jing@math.ncsu.edu}

\thanks{{\scriptsize
\hskip -0.4 true cm MSC (2010): Primary: 15A21; Secondary: 15A27.
Supported by Simons Foundation grant 198129 and NSFC grant 11271138.
}}

\maketitle

\begin{abstract} Two matrices $A$ and $B$ are called unitary (resp. orthogonal) equivalent if $AU=VB$ for two
unitary (resp. orthogonal) matrices $U$ and $V$. Using trace identities,
criteria are given for simultaneous unitary, orthogonal or complex orthogonal
equivalence between two sets of matrices. 
\end{abstract}

\section{Introduction}\label{section 1} \setcounter{equation}{0}

Let $\{A_1, A_2, \cdots, A_k\}$ and $\{B_1, B_2, \cdots, B_k\}$ be two sets of
complex matrices of size $m\times n$. We say that $\{A_i\}$ and $\{B_i\}$ are (simultaneous)
{\it unitary equivalent} if there exist two unitary matrices $U$, $V$ of respective dimension such that
\begin{align}\label{E:uni-equiv}
 V^*A_iU=B_i, \qquad i=1, \cdots, k,
 \end{align}
where $*$ means the transpose and conjugation. Two sets of real (resp. complex) matrices $\{A_i\}$, $\{B_i\}$ of the same size are said to be {\it orthogonal (resp. complex orthogonal) equivalent}
 if there are orthogonal (resp. complex orthogonal) matrices $P$ and $Q$ such that
 $Q^tA_iP=B_i$ for all $i$. When the matrices are square matrices and $U=V$, then
 two sets are said to be {\it (simultaneous) unitary similar}. We also
 say that $\{A_i\}$ are simultaneous {\it complex orthogonal similar} to $\{B_i\}$
 if $A_iO=OB_i$ for a complex matrix $O$ such that $OO^t=O^tO=I$.

 Simultaneous unitary similarity has been
 an important problem in representation theory.  When the set of matrices
 has no group structure, one of the first nontrivial results
 was given by Specht \cite{S2}, and the question has been studied by many people,
 see in particular \cite{F} for its history and
 difficulty. In \cite{GHS} algorithms are given to related problems on simultaneous unitary similarity or
 congruence for complex square matrices (see also \cite{AI}).
 Recently, important applications are found in quantum computation, where a key question
 of local unitary equivalence between two quantum states has been reduced to simultaneous orthogonal equivalence
 between two sets of real matrices (cf. \cite{JL}). Current resurgent interest
in quiver theory is also a reflection of the importance of the problem.

Geometrically (\ref{E:uni-equiv}) represents the matrix representation of a set of
linear transformations under common change of orthonormal bases in the domain and range
spaces. If we relax (\ref{E:uni-equiv}) to only require
that $A_iP=QB_i$ for two non-singular matrices $P, Q$
of respective dimensions, then $A_i$ and $B_i$ are two
matrix representations of a set of linear transformations $L_i$ under
different bases. The special case of two linear transformations ($k=2$) was answered by
Kronecker's theory \cite{G} of the matrix pencil $A_1+\lambda A_2$ using elementary divisors.
In further special cases of symmetric and antisymmetric matrices there are canonical forms under
similarity \cite{T}.
But for the corresponding unitary problem appropriate verifiable conditions of $P$ and $Q$ are needed.
In various applications in quantum computation, one is
concerned with the problem of how to judge simultaneous orthogonal
equivalence rather than finding the intertwining matrix, as the final solution relies on
the relevant problem in invariant theory.

In this note we will give criteria for (\ref{E:uni-equiv}) and its analogues
and show that two sets of complex
(real/complex) matrices are unitary (orthogonal/complex orthogonal) equivalent if and only if the corresponding traces of words in
$A_iA_j^*$ (resp. $A_iA_j^t$) are invariant. Usually the real version of a problem is harder, fortunately there is a satisfactory
solution for almost all of our statements thanks to a simple argument
to pass from the complex numbers to the real numbers. Our approach uses some results from semigroup theory.
As an application, we also obtain an alternative version of Albert's criterion
for simultaneous similarity of symmetric matrices (Remark \ref{r:jordan}), which
avoids the complicated special cases in the original argument.

We remark that a care is made to be self-contained and to
indicate the generalization to semigroups
from the corresponding results in finite groups.

\section{Square matrices}\label{section 2} \setcounter{equation}{0}

We first consider the {\it unitary similarity} (resp. {\it orthogonal similarity})
of two sets of complex (resp. real) square matrices.
Let $\{A_i\}$ be a set of complex (resp. real) matrices of same dimension.
We say that $\{A_i\}$
is {\it hermitian closed} (resp. {\it transpose closed})
 if each $A_i^*$ (resp. $A^t$) is contained in the span $\langle A_i\rangle$ of the set $\{A_i\}$. Two
 sets $\{A_i\}$ and $\{B_i\}$ of complex (resp. real)
 square matrices of equal dimension are said to be {\it unitary similar} (resp. {\it orthogonal similar})
 if there exists a unitary (resp. orthogonal) matrix $U$ such that $U^*A_iU=B_i$ for all $i$, or equivalently
 $U$ {\it intertwines $A_i$ and $B_i$} for all $i$.

Let $S$ be a set of matrices of equal dimension, a (product) {\it word $w(S)$ in the alphabet $S$} is a
matrix product $xy\cdots z$, where $x, y, \ldots, z$ are arbitrary matrices in $S$. We use $W(S)=\{w(S)\}$ to denote
the set of words in the alphabet $S$.

To give our first main result we need to
recall some basic notions of semigroups and their modules.
A set $G$ is a {\it semigroup} if there is a closed binary operation on $G$ that satisfies
associativity and has an identity. If every element of $G$ is invertible, then
$G$ becomes a group. We follow \cite[I]{J} to say that a vector space $V$ is
a $G$-module if there is an action of $G$ on $V$: $G\times V\ni(x, v)\mapsto x.v\in V$ such that $x.(y.v)=(xy).v$
for all $x, y\in G, v\in V$. We assume that all modules considered in this paper are finite dimensional left modules,
but we do not assume that the semigroup $G$ is finite. Equivalently
if $G$ is a (semi)group of linear transformations on $V$, then $V$ is a $G$-module
with the action given by the transformation. In this case, the entries of the matrix representation
are called the {\it coordinate functions} of $G$.

Two $G$-modules $V_i$ are called {\it equivalent} or {\it isomorphic}, denoted by $V_1\simeq V_2$,
if there is a linear map
$\phi:V_1\longrightarrow V_2$ such that $x.\phi(v)=\phi(x.v)$ for all $x\in G, v\in V$. In terms of the
matrix representation, this means that $xP=Px$ ($\forall x\in G\subset \mathrm{End}(V)$) for
some non-singular matrix $P$.

The notions of {\it reducible, completely reducible} and {\it decomposable}
$G$-modules can be defined as in the situation of group modules (cf. \cite[I]{J}).
The most useful ones for us are (i) $V$ is an irreducible $G$-module if $V$ has no non-trivial $G$-submodules.
It is known that the Schur lemma holds in this case, i.e., the only $G$-homomorphism between
two irreducible $G$-modules are scalar homomorphisms.
(ii) $V$ is a completely reducible $G$-module if
$V\simeq V_1\oplus\cdots\oplus V_r$ where $V_i$ are irreducible.

If $G$ is a semigroup of linear transformations on $V$, then the associated $G$-module $V$ is reducible
iff under the matrix representation, every element of $G$ is
uniformly of the block triangular form $x=\begin{bmatrix}x_1 & 0\\ x_3 & x_2\end{bmatrix}$ of
fixed shape, where $x_1$ and
$x_3$ are square matrices of size smaller than that of $x$; and $V$ is {\it decomposable} if all $x=\begin{bmatrix}x_1 & 0\\ 0 & x_2\end{bmatrix}$ with the fixed shape.

A $G$-module $V$ is called {\it unitary} if there exists a positive-definite hermitian form $(\ , \ )$ on $V$ such
that $(xu, xv)=(u, v)$ for any $x\in G$, and arbitrary $u, v\in V$. Equivalently
the corresponding linear transformations of the elements of $G$ are represented by unitary matrices.
In the case of a transpose-closed $G$-module,
this is equivalent to the existence of a $G$-invariant symmetric bilinear form on $V$.

The {\it character} $\chi(V)$ of $G$-module $V$ is the trace function $\chi(V)(x)=tr_V(x)$, $x\in G$.
We will show that two irreducible $G$-modules are equivalent iff their characters are the same
(Theorem \ref{T:Maschke}).

In finite group theory, every complex module is completely reducible (Maschke's theorem). For the semigroup theory,
this is not true in general, but we still have some form of Maschke's theorem for (possibly infinite) semigroups.

The following theorem was mainly due to Frobenius and Schur \cite{FS} 
for linear transformations. We remark that the statements hold for semigroups of
transpose closed linear transformations as well.

\begin{thm}\label{T:Maschke} Let $G$ be a semigroup.
(1) If $V$ is a unitary $G$-module, then $V$ is completely reducible.
(2) Two completely reducible $G$-modules are equivalent if and only if their characters are
equal.
\end{thm}
\begin{proof} (i) The idea of the proof is to show that
if $V$ is unitary reducible, then it is also unitary decomposable, i.e. the submodule has a complementary submodule
or it is a direct summand.
Using the matrix representation, this boils down to the fact that
a block triangular matrix must be
block diagonal if it is invariant under *-operation.

(ii) Since completely reducible modules
are direct sums of irreducible modules, by adding necessary irreducible summands with possibly zero
multiplicity, we can write the completely reducible modules as
\begin{align*}
V&=V_1^{\oplus a_1}\oplus \cdots \oplus V_r^{\oplus a_r}, \qquad a_i\in\mathbb Z_+\\
U&=V_1^{\oplus b_1}\oplus \cdots \oplus V_r^{\oplus b_r}, \qquad b_i\in\mathbb Z_+
\end{align*}
where $V_i$ are pairwise inequivalent irreducible
$G$-modules and assume that $\chi(V)=\chi(U)$.  Taking characters we have
\begin{align*}
(a_1-b_1)\chi(V_1)+\cdots +(a_r-b_r)\chi(V_r)=0.
\end{align*}
By a theorem of Frobenius and Schur \cite{FS} (cf. \cite[Th. 27.8]{CR}) the coordinate functions
of pairwise inequivalent irreducible modules for {\it semigroups} are linearly independent, their characters are
thus linearly independent. Therefore $a_i=b_i$ for $i=1, \ldots, r$, i.e. $V\simeq U$.
\end{proof}

We now come to our first result. It was announced in \cite{W}. 
The special case of a pair of matrices is the Specht criterion \cite{S2} for $\{A, A^*\}$ and $\{B, B^*\}$.
Since the basic results and notions
for semigroups are prepared above, a simple proof can be furnished as follows.

\begin{thm}\label{T:unitary} Let $\{A_i\}$ and $\{B_i\}$ be two sets of $n\times n$ hermitian closed complex matrices.
Then $\{A_i\}$ and $\{B_i\}$ are unitary similar if and only if $tr(w(\{A_i\})=tr(w(\{B_i\})$
for any word $w$.
\end{thm}
\begin{proof} Let the index set of $\{A_i\}$ be $I$, and consider the
free semigroup $G$ generated by $y_1, \cdots, y_{|I|}$.
The assignment $y_i\mapsto A_i$ (resp. $B_i$) defines a representations $V_A$ (resp. $V_B$) of $G$
on the $n$-dimensional space $V=\mathbb C^n$. Their associated $G$-modules are also denoted by $V_A$ and $V_B$
as well.

The hermitian closedness implies that both modules $V_A$ and $V_B$ are
unitary completely reducible modules for the free semigroup $G$. Since the trace of
every word in $G$ are the same, $V_A$ and $V_B$ have the same character
$\chi(V_A)=\chi(V_B)$. By Theorem \ref{T:Maschke} it follows that
$V_A\simeq V_B$, so there exists a non-singular matrix $P$ such that
\begin{align}
A_iP=PB_i, \qquad i=1, \cdots, k.
\end{align}
 Since $\{A_i\}$ and $\{B_i\}$ are hermitian closed, we also have
\begin{align}
P^*A_i=B_iP^*, \qquad i=1, \cdots, k.
\end{align}
Then $PP^*A_i=PB_iP^*=A_iPP^*$. As $PP^*$ is positive semi-definite, we let $(PP^*)^{1/2}$ be
the square root of $PP^*$. Then
$(PP^*)^{1/2}$ also commutes with $A_i$ for all $i$.
Let $U$ be the unitary part of $P$ in the polar decomposition.  Subsequently
$U=(PP^*)^{-1/2}P$ intertwines with $A_i$ and $B_i$ for all $i$, i.e. $A_iU=UB_i$ for $i$.
\end{proof}

It can be shown that only finitely many trace identities are needed. For
a pair of matrices, the bound of the word length is $n\sqrt{\frac{2n^2}{4(n-1)}+\frac14}+\frac{n-4}2$
\cite{P}. For sets of real symmetric matrices under Jordan-closeness,
the bound can be improved (see Remark \ref{r:jordan}).
The real version of the theorem also holds.
\begin{corollary}\label{C:sim}
Let $\{A_i\}$ and $\{B_i\}$ be two sets of real square matrices of the same size
and assume that both sets are closed under transpose.
Then $\{A_i\}$ and $\{B_i\}$ are orthogonal similar iff $tr(w(A_i))=tr(w(B_i))$ for any
word $w$ in respective alphabets, and iff there is a real matrix $P$
such that $PA_iP^{-1}=B_i$ for all $i$. 
\end{corollary}
\begin{proof} By Theorem \ref{T:unitary}, the trace identity implies
there exists a unitary matrix $U$ such that
$A_iU=UB_i$ for any $i$. Let $P$ and $Q$ be the real and imaginary part of $U$, then
one of them is nonsingular, say $P$. Taking the real part of the intertwining equation,
we obtain that $A_iP=PB_i$ for any $i$. Let $O$ be the orthogonal part of $PP^t$
in its polar decomposition, then the same argument of Theorem \ref{T:unitary} gives that
$A_iO=OB_i$ for all $i$.
\end{proof}

 If we replace the hermitian closeness condition by transpose closeness, then we
 have the following result.
\begin{thm} \label{T:orthogonal2} Let $\{A_i\}$ and $\{B_i\}$ be two sets of $n\times n$ transpose closed complex matrices. Then $\{A_i\}$ and $\{B_i\}$ are complex orthogonal similar if and only if $tr(w(\{A_i\})=tr(w(\{B_i\})$
for any word $w$ in respective alphabets.
\end{thm}
\begin{proof} We noted that Theorem \ref{T:unitary} holds for semigroups of matrices that are transpose closed.
The same proof of Corollary \ref{C:sim} can be repeated to derive the result.
\end{proof}

\begin{remark} The criteria do not hold without the condition of hermitian or transpose closeness.
For example let $A=\begin{bmatrix} 1 & 1\\ 0 &1 \end{bmatrix}$, $B=\begin{bmatrix} 1 & 0\\ 0 &1 \end{bmatrix}$.
Then $tr(A^n)=tr(B^n)$, but $A$ is not similar to $B$.
\end{remark}

The following application modifies Albert's criterion on simultaneous similarity of real symmetric matrices.

\begin{remark}\label{r:jordan} Let $\{A_i\}$ and $\{B_i\}$ be two sets of $k$ real symmetric matrices. Suppose
that the traces of any corresponding words in $A_i$ and $B_i$ are identical.
By a classical fact that any symmetric polynomial is a polynomial in the power-sum
symmetric polynomials, it follows that
$\det(xI-A)=\det(xI-B)$ for
any linear combination $A=\sum_{i=1}^k x_iA_i$ and $B=\sum_{i=1}^k x_iA_i$, where
$x, x_i$ are indeterminates. Albert's criterion \cite{A} says that
$\{A_i\}$ is ``almost'' simultaneous orthogonal similar to $\{B_i\}$
if the determinant identity holds plus that $\{A_i\}$ and $\{B_i\}$ are Jordan-closed, i.e. closed under anti-commutators $\{X, Y\}=XY+YX$ for any members
of the sets.
Here ``almost'' means that there are some counterexamples of degree 2 simple Jordan algebras of
dimension $4q_i+2$.
Moreover, \cite{A} showed that
these ``counterexamples'' can be removed if certain monomials in $\{A_i\}$ and $\{B_i\}$
of length $4q_i+2$
have the same traces. In particular, one only needs to verify the trace identity for word length
$\leq max(n, 4q_i+2)$ to ensure the simultaneous similarity of $\{A_i\}$ to $\{B_i\}$ under
the condition of Jordan-closeness.
\end{remark}

\section{Rectangular matrices}
The following elementary fact is needed for further discussion.
\begin{lemma}\label{L:norm} Let $a, b\in \mathbb C^n$ (resp. $\mathbb R^n$) be two column vectors. Then $aa^{*}=bb^{*}$
iff there is a complex number $\theta$, $|\theta|=1$ (resp. $\theta=\pm 1$) such that $a=\theta b$.
\end{lemma}
\begin{proof} Let $a=(a_1, \cdots, a_n)^t$ and $b=(b_1, \cdots, b_n)^t$. Clearly
$a_i=0$ iff $b_i=0$, so we can assume that
$a_1, \cdots, a_k$ are non-zero numbers, and $a_{k+1}=\cdots=a_n=0$. Since $|a_i|^2=|b_i|^2\neq 0$,
we can write $a_i=\theta_ib_i$ with $|\theta_i|=1$ for each $i=1, \cdots, k$. But
$a_i\overline{a_j}=b_i\overline{b_j}\neq 0$ imply that $\theta_i\overline{\theta_j}=1$ for any
$i, j=1, \cdots k$. Therefore $\theta_1=\cdots=\theta_k=\theta$.
Then we have that $a=\theta b$.
\end{proof}

We have the following generalization.
\begin{lemma}\label{L:matrixnorm} Let $A$, $B$ be two $m\times n$-matrices over $\mathbb C$ (resp. $\mathbb R$).
Then $AA^{*}=BB^{*}$ iff
there is a unitary (resp. orthogonal) matrix $V$
such that $A=BV$.
\end{lemma}
\begin{proof} The sufficient direction is clear. On the other hand,
let $U=[u_1, \cdots, u_m]$ be the unitary matrix of a basis of orthonormal eigenvectors
of $AA^*=BB^*$. Suppose the first $r$ eigenvectors have nonzero eigenvalues
$\sigma_i^2$, then $v_i=\sigma_i^{-1}A^*u_i$ exhaust all eigenvectors of $A^*A$
with non-zero eigenvalues by the singular value decomposition (same eigenvalue $\sigma_i^2$).
Extend $\{v_1, \cdots, v_r\}$ into a unitary matrix $V_1=[v_1, \cdots, v_n]$
of a basis of orthonormal eigenvectors
for $A^*A$.  Then $AV_1=UD, A^*U=V_1D^t$, where
$D=\begin{pmatrix} D_1 &0\\ 0 & 0\end{pmatrix}_{m\times n}$ and
$D_1=diag(\sigma_1, \cdots, \sigma_r)$. Similarly there exists another unitary
matrix $V_2$ such that
$BV_2=UD, B^*U=V_2D^t$. Subsequently $A=BV_2V_1^{-1}$.
\end{proof}

\begin{thm}\label{T:unitary2} Let $\{A_i\}$ and $\{B_i\}$ be two sets of complex (resp. real)
$m\times n$-matrices.
The following are equivalent.

(a) the set $\{A_i \}$
is unitary equivalent (resp. orthogonal equivalent) to the set $\{B_i\}$;

(b) the set $\{A_iA_j^* | i\leq j\}$
is unitary similar (resp. orthogonal similar) to the set $\{B_iB_j^* | i\leq j\}$;

(c) $tr\,w(\{A_iA_j^*\})=tr\,w(\{B_iB_j^*\})$ for any word $w$ in respective alphabets.
\end{thm}
\begin{proof} Equivalence of (b) and (c). Suppose (b) holds, then there is
a unitary matrix $U$ such that $U^*A_iA_j^*U=B_iB_j^*$ for any $i\leq j$. Taking
$*$ we also have $U^*A_jA_i^*U=B_jB_i^*$ for any $i\leq j$, then $U^*A_iA_j^*U=B_iB_j^*$ hold for
any $i, j$. Therefore the set $\{A_iA_j^*|i, j\}$ is simultaneous unitary equivalent to the set $\{B_iB_j^*|i, j\}$,
so (c) holds by Theorem \ref{T:unitary}. The converse direction is guaranteed by Theorem \ref{T:unitary}.

(b) clearly follows from (a). We now show that (b) implies (a). Let $A$ be the block matrix defined by $A^*=[A_1^*,\ \cdots, A_k^*]$. Then the block
matrix $AA^*=[A_iA_j^*]$ with $(i, j)$-entry being an $m\times m$ matrix $A_iA_j^*$.
By the argument above (b) implies that there is a unitary matrix $U$ such that
\begin{equation}
UA_iA_j^*U^*=B_iB_j^*
\end{equation}
for any $i, j$. Then
\begin{equation}
\begin{bmatrix} UA_1 \\ \vdots \\ UA_k\end{bmatrix}\begin{bmatrix} A_1^*U^* & \cdots & A_k^*U^*\end{bmatrix}
=\begin{bmatrix} B_1 \\ \vdots \\ B_k\end{bmatrix}\begin{bmatrix} B_1^* & \cdots & B_k^*\end{bmatrix}
\end{equation}
Using Lemma \ref{L:matrixnorm} for the block matrices $\begin{bmatrix} UA_1\\ \vdots\\ UA_k\end{bmatrix}$
and $\begin{bmatrix} B_1\\ \vdots\\ B_k\end{bmatrix}$, we get a unitary $m\times m$-matrix $V$ such that
\begin{align*}
\begin{bmatrix} UA_1 \\ \vdots \\ UA_k\end{bmatrix}
=\begin{bmatrix} B_1 \\ \vdots \\ B_k\end{bmatrix}V.
\end{align*}
Subsequently $UA_i=B_iV$ for all $i$.

The real case follows by a similar argument in view of Corollary \ref{C:sim}.
\end{proof}

The following result is clear from our discussion.

\begin{thm}\label{T:orthogonal3} Let $\{A_i\}$ and $\{B_i\}$ be two sets of complex
$m\times n$-matrices.
The following are equivalent.

(a) the set $\{A_i \}$
is complex orthogonal equivalent to $\{B_i\}$;

(b) the set $\{A_iA_j^t | i\leq j\}$
is complex orthogonal similar to $\{B_iB_j^t | i\leq j\}$;

(c) $tr\,w(\{A_iA_j^t\})=tr\,w(\{B_iB_j^t\})$ for any word $w$ in respective alphabets.
\end{thm}


\vskip30pt
\centerline{\bf ACKNOWLEDGMENTS}
The author thanks Professors M.~Putcha and V.~V.~Sergeichuk for interesting discussions.
He is also indebted to the referee for suggestions that have improved the paper.

\end{document}